\def\C             {{\ensuremath{\mathbb C}}}

\def\Q             {{\ensuremath{\mathbb Q}}}
\def\Z             {{\ensuremath{\mathbb Z}}}
\def\N             {{\ensuremath{\mathbb N}}}
\def\D             {{\ensuremath{\mathcal D}}}
\def\E             {{\ensuremath{\mathcal E}}}
\def\P             {{\ensuremath{\mathcal P}}}

\def\g             {{\ensuremath{\mathfrak g}}}

\newcommand\car[1] {\langle#1\rangle}
\def\df            {\,{:=}\,}

\def\eps           {\varepsilon}

\documentclass[a4paper, 11pt, reqno]{amsart}
\usepackage{amsmath, amsthm, amsfonts, amssymb} 
\usepackage{enumerate}
\usepackage{graphicx,color,pifont} 

\usepackage[colorlinks=false]{hyperref}
\usepackage[left=25mm,right=35mm, top=30mm, bottom=35mm,
includeheadfoot]{geometry}

\newtheorem{thm}{Theorem}
\newtheorem*{Thm}{Theorem}

\newtheorem{lem}[thm]{Lemma}

\theoremstyle{definition}
\newtheorem{expl}[thm]{Example}
\newtheorem{prob}[thm]{Problem}

\newtheorem{defi}[thm]{Definition}

\newtheorem{rem}[thm]{Remark}
\newtheorem{que}[thm]{Question}

\newcommand{\hamburger}[4] 
{
  \thispagestyle{empty}
  \vspace*{-2cm}
  \begin{flushright}
    ZMP-HH / #2 \\
    Hamburger Beitr{\"a}ge zur Mathematik Nr. #3 \\
    #4 \\
  \end{flushright}
  \vspace{0.5cm}
  \begin{center}
    \Large \bf
    #1
  \end{center}
  \vspace{0.5cm}
  \begin{center}	
    Simon Lentner and Daniel Nett \\
    Algebra and Number Theory, 
    Hamburg University,\\
    Bundesstra{\ss}e 55, D-20146 Hamburg \\
  \end{center}
  \vspace{0.5cm}
}

\begin{document}

\numberwithin{equation}{section}
\numberwithin{thm}{section}
                                     
\hamburger{A theorem on roots of unity and a combinatorial principle}%
          {14-20}{526}{September 2014}
%
%

\begin{abstract}
  Given a finite set of roots of unity, we show that all power sums are
  non-negative integers iff the set forms a group under multiplication. The
  main argument is purely combinatorial and states that for an arbitrary finite 
  set system the non-negativity of certain alternating sums is  equivalent to
  the set system being a filter.\\
  As an application we determine all discrete Fourier pairs of $\{0,1\}$-matrices.
  This technical result is an essential step in the classification of $R$-matrices
  of quantum groups.
\end{abstract}
\title{}
\date{}
\maketitle

\tableofcontents

\section{Introduction}
  In this paper we prove the following main theorem: 

  \begin{Thm}[\ref{thm:Roots}]
    Let $U$ be a non-empty finite set of complex roots
    of unity and consider the power sums
      $
        a_k\df\sum_{\zeta\in U} \zeta^{k}
      $.
    Then all $a_k$ are non-negative integers iff $U$ is actually a
    multiplicative group of roots of unity (i.e. all $n$-th roots of unity for
    some $n$).
  \end{Thm}

  The proof of the theorem is combinatorial in nature. Especially if the order
  of all $\zeta$ are squarefree numbers and hence correspond to
  subsets of primes, the statement amounts to the following apparently new
  combinatorial principle, which is interesting in its own right: 

  \begin{Thm}[\ref{thm:Filter}]
    Let $N$ be finite set, $\P(N)$ denote the power set of $N$ and $\E\subset
    \P(N)$. Let $\mu\colon\P(N)\to[0,\infty]$ be a measure on $\P(N)$.  Then the
    following is equivalent:
    \begin{enumerate}[(i)]
      \item $a_{C}:=(-1)^{|N|} \sum_{D\in\E} (-1)^{|C\cup D|}
        e^{\mu(C-D)}\geq 0$ for all $C\subset N$.
      \item $\E=\{D\subset N\mid D\supset A\}$ for some $A\subset N$. Such a set
        $\E$ is called a filter in $N$.
    \end{enumerate}
  \end{Thm}

  The proof of the main theorem proceeds along these lines with $\mu$ some
  explicit number theoretic function. But to include the non-square-free case
  the set system $\E\subset \P(N)$ is roughly replaced by the set of orders of
  $\zeta\in U$, which is partially ordered via divisibility. It
  would be nice to prove the main theorem even more generally for any 
  partially ordered set. Note that the expressions we calculate remind very
  strongly on partition functions in statistical physics.\\

  We briefly discuss the proof strategy: We perform an
  induction on the greatest common multiple
  $N$ of the orders of the $\zeta\in U$. Given the set of numbers
  $\E=\{N/ord(\zeta)\;|\;\zeta\in U\}$ we define sets $\E_p$ by decreasing the
  power of a prime $p$ in each number and removing non-divisible ones
  (Definition
  \ref{defi:Ep}). In Lemma \ref{lem:acp} we show that the assumption
  of $a_k\geq 0$ for $\E$ implies it also for all $\E_p$. In Lemma
  \ref{lem:induction} we use the induction hypothesis that all $\E_p$ are
  filters to show that $\E_p$ is almost a filter. Since an explicit
  calculation in Lemma \ref{lem:modification} has shown that small
  modifications of a filter usually violate the condition $a_k\geq 0$ for some
  $k$ we see that $\E$ is actually a filter.\\

  \noindent
  As an application we prove the following theorem: 
  \begin{Thm}[\ref{thm:eps-Solutions}]
    All idempotents $\eps/N$ of the group algebra $\C[\Z_N\times\Z_N]$ with
    $\eps_{ij}\in\{0,1\}$, or equivalently all discrete Fourier
    pairs $\eps,\bar\eps$ of $\{0,1\}$-matrices are either
    \begin{equation}
      \eps_{ij} = \delta_{(\frac Nd\mid i)} \delta_{(d\mid j-t\frac i{N/d})},
    \end{equation}
    for a unique $d\mid N$ and $0\leq t\leq d-1$ or
    they are trivial $\eps=\bar\eps=0$. 
  \end{Thm}

  The significance of this technical result is the classification of
  $R$-matrices for quantum groups and hence for constructing certain braided
  categories. Lusztig's ansatz for such $R$-matrices \cite{Lus93} Sec. 32.1,
  contains a free parameter $R_0\in\C[\Lambda\times \Lambda]$ for some abelian
  group $\Lambda$ and a system of equations on $R_0$. In the last section of
  this paper, using the previous theorem, we will solve a subset of these
  equations only depending  on an abelian group.

  Once these explicit solutions have been obtained, they can be plugged into the
  remaining equations which depend heavily on the specific parameters of
  the quantum group. This is done in a rather Lie-theoretic case-by-case
  argument in \cite{LN14}.
 
  \paragraph{\bf Acknowledgements} Partly supported by the DFG Priority Program
  1388 ``Representation theory''. We thank Christian Reiher for several helpful
  comments.


\section{A combinatorial principle}\label{sec:Filter}
  Before we turn to the proof of the main Theorem \ref{thm:Roots} we prove the
  following combinatorial principle. It shows that the main Theorem does not
  depend on specific properties of prime numbers, but is cominatorial in
  nature. It also gives the blueprint for the proof of the main theorem.
  
  \begin{thm}\label{thm:Filter}
    Let $N$ be finite set, $\P(N)$ denote the power set of $N$ and $\E\subset
    \P(N)$.
    Let $\mu\colon\P(N)\to[0,\infty]$ be a measure on $\P(N)$.  Then the
    following is equivalent:
    \begin{enumerate}[(i)]
      \item $a_{C}:=(-1)^{|N|} \sum_{D\in\E} (-1)^{|C\cup D|}
        e^{\mu(C-D)}\geq 0$ for all $C\subset N$.
      \item $\E=\{D\subset N\mid D\supset A\}$ for some $A\subset N$. Such a set
        $\E$ is called a \emph{filter} in $N$ (see e.g. \cite{Bou66} \S 6).
    \end{enumerate}
  \end{thm}
  \noindent
  The remainder of this section is devoted to the proof of this theorem. \\
  
  A straightforward calculation gives the values of the $a_C$ if $\E$ is a
  filter. It shows immediately the implication $(i)\to(ii)$, but the precise
  value will also be crucial to the proof of the converse in what follows:
  
  \begin{lem}\label{lem:comb_Values-ac}
      Let $\E$ be a filter, i.e. $\E=\{D\subset N\mid D\supset A\}$ for some
      $A\subset N$. Then for any $C\subset N$ we have
      $$a_{C}:=(-1)^{|N|} \sum_{D\in\E} (-1)^{|C\cup D|}
        e^{\mu(C-D)}=\begin{cases}
	  e^{\mu(N-A)}\prod_{p\in N-A}\left(1+e^{-\mu(p)}\right), & C\cup A=N\\
	  0, & else
        \end{cases}$$
  \end{lem}
  \begin{proof}
  \begin{align*}
      a_{C}
      &=(-1)^{|N|} \sum_{A\subset D\subset N} (-1)^{|C\cup D|}e^{\mu(C-D)}\\
      &=(-1)^{|N|} \sum_{D'\subset N-A} (-1)^{|C\cup A\cup D'|}e^{\mu(C-A-D')}\\
      &=(-1)^{|N-A|} \sum_{D'\subset N-A} (-1)^{|(C-A)\cup D'|}e^{\mu(C-A-D')}
  \end{align*}
  This shows that the value of $a_C$ for the filter generated by $A$ in $N$ is
  equal to the value of $a_{C-A}$ for the filter generated by $\varnothing$ in
  $N-A$. Thus is suffices to show the claim for the filter $\E=\P(N)$ generated
  by $A=\varnothing$:
  \begin{align*}
      a_{C}
      &=(-1)^{|N|} \sum_{D\subset N} (-1)^{|C\cup D|}e^{\mu(C-D)}\\
      &=(-1)^{|N|} \sum_{D_1\subset C,D_2\subset N-C}
	(-1)^{|C|+|D_2|}e^{\mu(C)-\mu(D_1)}\\
      &=(-1)^{|N|+|C|}e^{\mu(C)} \left(\sum_{D_2\subset N-C}
	(-1)^{|D_2|}\right)\left(\sum_{D_1\subset
	C} e^{-\mu(D_1)}\right)\\
      &=(-1)^{|N|+|C|}e^{\mu(C)}\left(\prod_{p\in N-C} \left(1-1\right)\right) 
	\left(\prod_{p\in C}\left(1+e^{-\mu(p)}\right)\right)\\
      &=\begin{cases}
	  e^{\mu(N)}\prod_{p\in N}\left(1+e^{-\mu(p)}\right), & C=N\\
	  0, & else
        \end{cases}
  \end{align*}
  The general formula for arbitrary $A$ follows by again replacing $N$ with
  $N-A$ and $C$ by $C-A$.
  \end{proof}

  We use this result to show that if $\E$ is a small modification of a
  filter, the main assumption $a_C\geq 0$ for all $C\subset N$ usually
  fails to be true.
  \begin{lem}\label{lem:comb_modification}~
    \begin{enumerate}[(a)]
      \item Let $\E\neq\P(N)$ be a filter in $|N|>1$,
        then $\E\cup\{\varnothing\}$ gives $a_C<0$ for some $C\subset N$.
      \item Let $\E=\P(N)$ in $|N|>1$, then
        $\E\setminus\{\varnothing\}$ gives $a_C<0$ for some $C\subset N$. 
    \end{enumerate}
    Note that on the other hand for $|N|=1$ and $\E$ the only filter
    $\E\neq\P(N)$ we have that both $\E\cup\varnothing$ and $\P(N)-\varnothing$
    are filters (namely $\P(N)$ and $\E$).
  \end{lem}

  \begin{proof}
    \begin{enumerate}[(a)]
      \item By assumption $\E$ is a filter generated by some $A\neq
      \varnothing$ for $|N|>1$. We wish to find a negative value
      of some $\tilde a_{C}$ for the
      set system $\tilde{\E}:=\E\cup\{\varnothing\}$: Suppose first that also
      $A\neq N$ and choose some $p\in N-A$, then $a_{N-p}=0$ by Lemma 
      \ref{lem:comb_Values-ac} and thus:      
        \begin{align*}
	  \tilde{a}_{N-p}
	  &=(-1)^{|N|} \sum_{D\in\E\cup \{\varnothing\}}
	  (-1)^{|C\cup D|} e^{\mu(C-D)}\\
	  &=a_{N-p}+(-1)^{|N|+|N-p|}e^{\mu(N-p)}\\
	  &=-e^{\mu(N-p)}<0
        \end{align*}
        Suppose now that $A=N$ and choose some $q\in N$, then again by Lemma 
	\ref{lem:comb_Values-ac}:
	\begin{align*} 
	  \tilde{a}_{N-q}
	  &=(-1)^{|N|} \sum_{D\in\E\cup \{\varnothing\}}
	  (-1)^{|C\cup D|} e^{\mu(C-D)}\\
	  &=a_{N-q}+(-1)^{|N|+|N-q|}e^{\mu(N-q)}\\
	  &=1-e^{\mu(N-q)}<0
	\end{align*}
      \item By assumption $\E=\P(N)$ for $|N|>1$, so $\E$ is the filter
      generated by $A=\{\varnothing\}$. Then again by Lemma
      \ref{lem:comb_Values-ac} $a_C=0$ for $C\neq N$. Choose any $q\neq p\in N$,
      then we calculate $\tilde{a}_{N-\{p,q\}}$ for the filter
      $\tilde{\E}:=\E-\{\varnothing\}$:
	\begin{align*} 
	  \tilde{a}_{N-\{p,q\}}
	  &=(-1)^{|N|} \sum_{D\in\E-\{\varnothing\}}
	  (-1)^{|C\cup D|} e^{\mu(C-D)}\\
	  &=a_{N-q}-(-1)^{|N|+|N-\{p,q\}|}e^{\mu(N-\{p,q\})}\\
	  &=1-e^{\mu(N-\{p,q\})}<0
	\end{align*}
      \end{enumerate}
  \end{proof}
  We now proceed by introducing the induction step along $|N|$:
    \begin{defi}\label{defi:comb_Ep}
      Let $\E$ be any set system in $N$ and $p\in N$, then we define a new set
      system for $N-p$ by
      \begin{equation*}
        \E_p=\{D-p\mid p\in D,D\in\E \}
      \end{equation*}
      For $C\subset N-p$ we denote by $a_C^p$ the corresponding sum over $\E_p$,
      i.e.
      \begin{equation*}
        a_C^p =(-1)^{|N-p|} \sum_{D\in\E_p}(-1)^{|C\cup D|} e^{\mu(C-D)}
      \end{equation*}
      We will in the following only consider $\E_p$ for all $p$, such that
      there exists any $D\in \E$ with $p\in D$, so $\E_p$ is not empty.
    \end{defi}

    We first wish to prove that our main assumption $a_C\geq 0$ implies
    $a_C^p\geq 0$ in $\E_p$:
    \begin{lem}\label{lem:comb_acp}
      For any $p\in N$ we get for all $C\in\E_p$ (note that $p\not\in C$):
      \begin{equation*}
        a_C^p =
	\frac{e^{\mu(p)}}{1+e^{\mu(p)}}a_C+\frac{1}{1+e^{\mu(p)}}a_{C\cup p}
      \end{equation*}
      In particular, $a_C\geq 0$ for all $C\in\E$ implies $a_C^p\geq 0$ for
      all $C\in \E_p$.
    \end{lem}
  \begin{proof}
    We calculate the right hand side by splitting the sum over all $D\in \E$
    into two summands for all $p\not\in D$ resp. $p\in D$ and use $p\not\in C$.
    The latter set of $D$ then correspond to $D'=D-p$ in $\E_p$:
    \begin{align*}
      &\frac{e^{\mu(p)}}{1+e^{\mu(p)}}a_C+\frac{1}{1+e^{\mu(p)}}a_{C+p}\\
      &=(-1)^{|N|} \frac{e^{\mu(p)}}{1+e^{\mu(p)}}
      \sum_{p\in D\in\E}(-1)^{|C\cup D|}e^{\mu(C-D)}
      +(-1)^{|N|} \frac{1}{1+e^{\mu(p)}}
      \sum_{p\in D\in\E}(-1)^{|C\cup p \cup D|}e^{\mu((C\cup p)-D)}\\
      &+(-1)^{|N|} \frac{e^{\mu(p)}}{1+e^{\mu(p)}}
      \sum_{p\not\in D\in\E}(-1)^{|C\cup D|}e^{\mu(C-D)}
      +(-1)^{|N|} \frac{1}{1+e^{\mu(p)}}
      \sum_{p\not\in D\in\E}(-1)^{|C\cup p\cup D|}e^{\mu((C\cup p)-D)}\\
      &=(-1)^{|N|} \frac{e^{\mu(p)}}{1+e^{\mu(p)}}
      \sum_{p\in D\in\E}(-1)^{|C\cup D|}e^{\mu(C-D)}
      +(-1)^{|N|} \frac{1}{1+e^{\mu(p)}}
      \sum_{p\in D\in\E}(-1)^{|C\cup D|}e^{\mu(C-D)}\\
      &+(-1)^{|N|} \frac{e^{\mu(p)}}{1+e^{\mu(p)}}
      \sum_{p\not\in D\in\E}(-1)^{|C\cup D|}e^{\mu(C-D)}
      +(-1)^{|N|} \frac{1}{1+e^{\mu(p)}}
      \sum_{p\not\in D\in\E}(-1)^{|C\cup D|+1}e^{\mu(C-D)+\mu(p)}\\
      &=\left(\frac{e^{\mu(p)}}{1+e^{\mu(p)}}+\frac{1}{1+e^{\mu(p)}}\right)
      \cdot (-1)^{|N|} \sum_{p\in D\in\E}(-1)^{|C\cup D|}e^{\mu(C-D)}\\
      &+\left(\frac{e^{\mu(p)}}{1+e^{\mu(p)}}
	-\frac{1}{1+e^{\mu(p)}}\cdot e^{\mu(p)}\right)\cdot 
      (-1)^{|N|}\sum_{p\not\in D\in\E}(-1)^{|C\cup D|}e^{\mu(C-D)}\\
      &=(-1)^{|N|} \sum_{p\in D\in\E}(-1)^{|C\cup D|}e^{\mu(C-D)}
      =(-1)^{|N-p|} \sum_{D'\in\E_p}(-1)^{|C\cup D'|}e^{\mu(C-D')}=a_C^p\\
    \end{align*}
  \end{proof}

  Thus if all $a_C\geq 0$ by induction hypothesis all $\E_p$ are filters.
  We now conclude the  induction that $\E$ is  a filter if all possible
  reductions $\E_p$ are  filters. As induction step, we use the following lemma.

  \begin{lem}\label{lem:comb_induction}
    Let $\E$ be a set system for $N$ such that all $\E_p$ are
    filters generated by sets $A_p\subset N-p$. Then either there exists a
    $p\in N$ with $p\in A_q$ for all $p\neq q$ or for all $p\in N$ we have
    $A_p=\varnothing$.\\
    In the first case we show that $\E$ is the filter generated by $p\cup A_p$
    or $\E$ is the set system consisting of this filter together with
    $D=\varnothing$. In the second case we show $\E=\P(N)$ or
    $\E=\P(N)-\varnothing$.
  \end{lem}

  \begin{proof}
    Assume there exists $q'$ with $A_{q'}\neq \varnothing$ and let $p\in
    A_{q'}$, then we claim $p\in A_q$ for all $q\neq p$. We prove this by
    contradiction, since if    $p\not\in A_q$ for some $q$ then we consider
    $q'\cup A_q\in \E_q$ (since    $\E_q$ is a filter) and hence $q\cup q'\cup
    A_q\in \E$ (by definition of    $\E_q$). But then $q\cup A_q \in\E_{q'}$ and
    $A_{q'}\subset q\cup A_q$    (since $\E_{q'}$ is a filter). But this
    contradicts $p\not\in A_q$, which    shows the first part of the Lemma.

    We now prove the consequences in the two cases. In the first case we assume
    it
    exists $p\in A_q$ for all $q\neq p$. Let $\D\supset p\cup A_p$ then $D\in
    \E_p$ (since $E_p$ is a filter) and $D\in \E$ (by definition of $\E_p$).
    Let now conversly by $D\in \E$. If $p\in D$ then we have $D-p\in\E_p$ (by
    definition of $\E_p$) and hence also $D\supset A_p$ (since $\E_p$ is a
    filter), implying $\D\supset p\cup A_p$. If $p\not\in D$ then either
    $D=\varnothing$ or some $q\in D$. In the latter case $D-q\in\E_q$, hence
    $D-q\supset A_q\ni p$ which is a contradicion. So either $\D\supset p\cup
    A_p$ or $D=\varnothing$ as asserted.\\
    In the second case we assume $A_p=\varnothing$ for all $p\in N$, hence any
    for any set $D\neq \varnothing$ we may chose some $p\in D$ and yield
    $D-p\in\E_p$ and hence $D\in \E_q$. Hence any set with the possible
    exception of $D=\varnothing$ is in $\E$ as asserted.
  \end{proof}

  We can now conclude the inductive proof of the implication (i)$\to$(ii)
  in Theorem \ref{thm:Filter}: 
  For $|N|=0$ the only set system is $\E=\{\varnothing\}$ and is a filter. Let
  $|N|\geq 1$ and $\E$ such that all $a_C\geq 0$, then $a_C^p\geq 0$ for all 
  $p\in N$ by Lemma \ref{lem:comb_acp}. Thus by induction hypothesis all $\E_p$ 
  are filters. Then by Lemma \ref{lem:comb_induction} we have that either $\E$ 
  is a filter (in which case the induction step is finished) or some filter 
  $\E\neq \P(N)$ together with $\varnothing$ or $\E=\P(N)-\{\varnothing\}$. By
  Lemma \ref{lem:comb_modification} these two cases can only fulfill $a_C\geq 0$
for  $|N|=1$ where both are filters. This concludes the proof
  of Theorem \ref{thm:Filter}.

\section{A theorem about roots of unity}\label{sec:Roots}
  \begin{thm}\label{thm:Roots}
    Let $U$ be a non-empty finite set of complex roots
    of unity and consider the power sums
      $
        a_k\df\sum_{\zeta\in U} \zeta^{k}
      $.
    Then all $a_k$ are non-negative integers iff $U$ is actually a
    multiplicative group of roots of unity (i.e. all $n$-th roots of unity for
    some $n$).
  \end{thm}
  \noindent
  The remainder of this section is devoted to the proof of this theorem. \\

  Since $U$ is finite, we may assume some integer $N$
  such that $U\subset \Sigma_N=\{\zeta\in\C\mid \zeta^N=1\}$. Let $\xi_N$ the
  primitive $N$-th root of unity $\exp(2\pi i/N)$. We start
  with the observation, that the set $U$ is a union of Galois orbits of ${\rm
  Gal}(\xi_N)$ acting on $\Sigma_N$. \\
  In the following, we denote by $(a,b)$ the greatest common divisor of
  two integers $a,b$. 

  \begin{lem}\label{lem:Orbits}
    Any $U$ as in Theorem \ref{thm:Roots} is invariant
    under the Galois group $G={\rm Gal}(\xi_N)$, i.e. it is a union of orbits
    of $G$ acting on $\Sigma_N$. Each orbit consist of all primitive 
    roots of unity for some divisor of $N$ and hence $a_k$ only depends on
    $(k,N)$.
  \end{lem}

  \begin{proof}
    Let $p(x)=\prod_{\zeta\in U} (x-\zeta)\in\C[x]$, i.e. $p(\zeta)=0$ for all
    $\zeta\in U$. Denote $t=|U|$ and $U=\{\zeta_1,\ldots,\zeta_t\}$. For
    $0\leq k\leq t$ let $\sigma_k(x_1,\ldots,x_t)=\sum_{1\leq j_1<\ldots<
    j_k\leq t}x_{j_1}\cdot \ldots \cdot x_{j_k}$ be the elementary symmetric
    polynomials.  Then $p(x)=\sum_{k=0}^t (-1)^{t-k}
    \sigma_{t-k}(\zeta_1,\ldots,\zeta_t) x^k$. Let
    $s_k(x_1,\ldots,x_t)=\sum_{i=1}^t x_i^k$, then we have in particular,
    $a_k=s_k(\zeta_1,\ldots,\zeta_t)$. By the Newton identities, the
    $\sigma_k(x_1,\ldots,x_t)$ can be expressed as sums of powers of the $s_k$
    with rational coefficients, e.g. $\sigma_2 = \frac 12 s_1^2 - \frac12
    s_2$. Thus, we have that the coefficients of $p(x)$, the
    $\sigma_k(\zeta_1,\ldots,\zeta_k)$, are sums of integers with rational
    coefficients, hence $p(x)\in\Q[x]$. (In fact, we have $p(x)\in\Z[x]$,
    since the $\sigma_k(\zeta_1,\ldots,\zeta_t)$ are algebraic integers in
    $\Q$, hence in $\Z$.) Thus we get, that the Galois group $G$
    permutes the roots of $p(x)$, i.e. $U$ consists of orbits of $G$.
  \end{proof}

  \begin{defi}\label{defi:Filter}
    Let $N\in\N$. The set $\D(N)=\{d\in\N\mid\, d\mid N\}$ is the \emph{set of
    all divisors of $N$}. We call a set $\E\subset \D(N)$ a \emph{filter in
    $\D(N)$} if there exist an $e\mid N$ such that $\E=e\D(N/e)=\{d\mid
    N\,\mid \, e\mid d\}$. In this case we write $\E=(e)_N$ or shortly $(e)$
    for the filter in $\D(N)$.
  \end{defi}

  By Lemma \ref{lem:Orbits}, the set $U$ is of the form
  $U=\bigcup_{d\in\E}\{\xi_N^i\mid (N,i)=d\}$ for a set
  $\E\subset\D(N)$. We wish to prove that $\E$ is a filter and hence $U$ is a
  subgroup. For $c\mid N$ we have
  \begin{equation*}
    a_c=\sum_{\zeta\in U}\zeta^c 
    = \sum_{d\in\E} 
    \sum_{(i,N)=d} \xi_N^{ic}.
  \end{equation*}
  
  A straightforward calculation gives the values of the $a_c$.
  \begin{lem}
    For $N\in\N$ and $c\in \D(N)$, we have
    \begin{equation*}
      a_c = \sum_{d\in\E} \frac{\varphi(N/d)}{\varphi(N/(N,dc))}\mu\left(\frac
      N{(N,dc)}\right).
    \end{equation*}
    Here, $\varphi\colon\N\to\N$ is the Euler $\varphi$-function, given by
    $\varphi(\prod_{i=1}^tp_i^{r_i}) = \prod_{i=1}^t (p_i-1)p_i^{r_i-1}$ for
    mutually different prime numbers $p_i$, and $\mu\colon\N\to\{-1,0,1\}$ is
    the Moebius function, defined by $\mu(n)=1$ if $n$ is square-free and has
    an even number of prime factors, $\mu(n)=-1$ if $n$ is square-free and has
    an odd number of prime factors and $\mu(n)=0$ if $n$ has a squared prime
    factor.
  \end{lem}

  \begin{proof}
    It is an elementary number theoretical fact, that for an primitive $N$-th
    root of unity $\xi$ we have
      $$\sum_{\substack{i=1 \\ (i,N)=1}}^N \xi^i=\mu(N)$$
    with the Moebius function $\mu$. For $d\mid N$ we have
      $$
        \sum_{\substack{i=1 \\ (i,N)=d}}^N \xi_N^i=
        \sum_{\substack{i=1 \\ (i,N/d)=1}}^{N/d} \xi_{N/d}^i= 
        \mu(N/d)
      $$
    with primitive $(N/d)$-th root of unity $\xi_{N/d}$. For $c\mid N$ we
    get
      $$
      \sum_{\substack{i=1 \\ (i,N)=d}}^N \xi^{ic}=
      \sum_{\substack{i=1 \\ (i,N/d)=1}}^{N/d} \xi_{N/d}^{ic}= 
      \sum_{\substack{i=1 \\ (i,N/d)=1}}^{N/d} \xi_{N/(N,dc)}^{i}= 
      \frac{\varphi(N/d)}{\varphi(N/(N,dc))}\mu\left(\frac N{(N,dc)}\right),
      $$
    since the last sum has $\varphi(N/d)$ summands which contain
    $\varphi(N/(N,dc))$-times all primitive $N/(N,dc)$-th roots of unity
    and their sum gives $\mu(N/(N,dc))$.
  \end{proof}

  Next, we calculate the $a_c$ explicitly in the case $\E$ is a filter in
  $\D(N)$.
  \begin{lem}\label{lem:Values-ac}
    Let $c\mid N$ and $\E=(e)$ be a filter for some $e\mid N$. Then
    \begin{equation*}
      a_c = \begin{cases}
        N/e, & c\in (N/e)_e=(N/e)\D(e), \\
        0,   & \text{else}.
            \end{cases}
    \end{equation*}
    Especially, for a $\E$ being a filter, we have $a_c\geq 0$ for all
    $c\in\D(N)$.
  \end{lem}

  \begin{proof}
    We calculate $a_c$ for all $c\in\D(N)$:
    \begin{align*}
      a_c &= \sum_{d\in\E} \frac{\varphi(N/d)}{\varphi(N/(N,dc))}\mu\left(\frac
             N{(N,dc)}\right)\\
          &= \sum_{d'\in\D(N/e)}
             \frac{\varphi(N'e/d'e)}{\varphi(N'e/(N'e,d'ec))}\mu\left(\frac
             {N'e}{(N'e,d'ec)}\right) \tag{$N'=N/e,~ d'=d/e$}\\
          &= \sum_{d'\in\D(N/e)}
             \frac{\varphi(N'/d')}{\varphi(N'/(N',d'c))}\mu\left(\frac
             {N'}{(N',d'c)}\right) \\
    \end{align*}
    Thus, we can assume $\E=\D(N)$ and omit the superscript $'$.
    Since $\varphi$ and $\mu$ are multiplicative functions, we may assume
    $N=p^{N_p}$, $N_p>0$, for a prime $p$ and $d=p^{d_p}$, $c=p^{c_p}$ for
    $d,c\mid N$ and $0\leq d_p,c_p\leq N_p$. Then
    \begin{align*}
      a_c &= \sum_{d\in\D(N)}
      \frac{\varphi(N/d)}{\varphi(N/(N,dc))}\mu\left(\frac N{(N,dc)}\right)\\
          &= \sum_{d\in\D(N)}
      \frac{\varphi(p^{N_p-d_p})}{\varphi(p^{N_p-\min\{d_p+c_p,N_p\}})}
      \mu\left(p^{N_p-\min\{d_p+c_p,N_p\}}\right) \\
      &= \sum_{i=0}^{N_p}
      \frac{\varphi(p^{N_p-i})}{\varphi(p^{N_p-\min\{i+c_p,N_p\}})}
      \mu\left(p^{N_p-\min\{i+c_p,N_p\}}\right)
    \end{align*}
    Since the $\mu$-term equals $0$ if $i+c_p<N_p-1$, is equal to $-1$ if
    $i+c_p=N_p-1$ and $+1$ otherwise, we get
    \begin{align*}
      a_c &= \sum_{i=0}^{N_p-1}(p-1)p^{N_p-i-1} \, + 1\\
          &= (p-1) p^{N_p-1} \frac{p^{-N_p}-1}{p^{-1}-1} \,+1\\
          &= p^{N_p}-1+1 = p^{N_p} = N,
    \end{align*}
    for $c_p=N_p$, and
    \begin{align*}
      a_c &= \sum_{i=N_p-c_p-1}^{N_p-1}
             \frac{\varphi(p^{N_p-i})}{\varphi(p^{N_p-\min\{i+c_p,N_p\}})}
             \mu\left(p^{N_p-\min\{i+c_p,N_p\}}\right) \, +1 \\
          &= -\frac{(p-1)p^{N_p-(N_p-c_p-1)-1}}
             {(p-1)p^{N_p-(N_p-c_p-1+c_p)-1}}
             + \sum_{i=N_p-c_p}^{N_p-1}(p-1)p^{N_p-i-1} \, + 1\\
          &= -p^{c_p} + (p-1) p^{N_p-1}p^{-(N_p-c_p)} \sum_{i=0}^{c_p-1}
             p^{-i} \,+1 \\
             &= -p^{c_p} - (1-p^{c_p}) + 1 = 0,   
    \end{align*}
    for $0\leq c_p<N_p$. Thus, in the general case $\E=(e)$, we have
    $a_c=N'=N/e$ for $c=N/e$ and all multiples, hence the lemma is proven.
  \end{proof}

  We use this result to show that if $\E$ is a small modification of a
  filter, the main assumption $a_c\geq0$ for all $c\in\D(N)$ usually fails to
  be true.
  \begin{lem}\label{lem:modification}~
    \begin{enumerate}[(a)]
      \item Let $\E=(e)\neq\D(N)$ be a filter, but not $(p)$ for $N=p^n$ a
        prime power.
        Then $\E\cup\{1\}$ gives $a_c<0$ for some $c$. 
        If $N=p^n$ for some prime number $p$, $n\in\N$, and $\E=(p)$ a
        filter, then $\E\cup\{1\}$ is a filter as well.
      \item Let $\E=(1)=\D(N)$ and $N$ be not a prime power. Then
        $\E\setminus\{1\}$ gives $a_c<0$ for some $c$. In the case $N=p^n$,
        the set $\E\setminus\{1\}$ is a also a filter, namely $(p)$.
    \end{enumerate}
  \end{lem}

  \begin{proof}
    \begin{enumerate}[(a)]
      \item Assume at first, that $N=\prod_{p\mid N}p^{N_p}$, $N_p\geq 1$ for
        all $p$, is not a prime power. If $p\nmid e$ for a prime divisor
        $p\mid N$, we have $N/e \nmid N/p$, and therefore $a_{N/p} =0$ by
        Lemma \ref{lem:Values-ac}.  We calculate the value $\tilde a_{N/p}$
        for $\E\cup\{1\}$:
        \begin{align*}
          \tilde a_{N/p}  &=\sum_{d\in\E\cup\{1\}} 
          \frac{\varphi(N/d)}{\varphi(N/(N,d(N/p)))} \mu(N/(N,d(N/p))) \\
          &= a_{N/p} + 
          \frac{\varphi(N)}{\varphi(p)} \mu(p)
          = 0 \,+ \frac{\varphi(N)}{p-1} (-1) < 0.
        \end{align*}
        Let $e=\prod_{p\mid N} p^{e_p}$ with primes $p$. If $e_p\geq 1$ for
        all $p$ we have $N/e \leq N/(\prod_{p\mid N} p) = \prod_{p\mid N}
        p^{N_p-1}$. We calculate $\tilde a_{N/q}$ for some prime divisor $q$
        and $\E\cup\{1\}$:
        \begin{align*}
          \tilde a_{N/q}  
          &= a_{N/q} + \frac{\varphi(N)}{\varphi(q)} \mu(q) \\
          &\leq \prod_{p\mid N}p^{N_p-1} - \frac{\prod_{p\mid N}
                (p-1)p^{N_p-1}}{q-1} \\
          &= \prod_{p\mid N} p^{N_p-1} \left(1-\prod_{p\neq q} (p-1) \right)
           < 0.
        \end{align*}

        Assume now, $N=p^n$ and $e=p^k$ for $1< k\leq n$, thus $a_c=p^{n-k}$
        for all $c=p^{n-k+l}$ for $0\leq l\leq k$. We calculate $\tilde
        a_{p^{n-1}}$ for $\E\cup\{1\}$:
        \begin{align*}
          \tilde a_{p^{n-1}} =
          a_{p^{n-1}} + \frac{\varphi(p^n)}{\varphi(p)}\mu(p) = p^{n-k} -
          p^{n-1} <0.
        \end{align*}
        If $N=p^n$ and $\E=(p)$, we have $\E\cup\{1\}=\D(N)$, hence it is
        a filter.

      \item If $\E=\D(N)$, it is $a_c=N$ for $c=N$ and $0$ otherwise by Lemma
        \ref{lem:Values-ac}. Since $N$ is not a prime power, there exist
        distinct primes $p,q\mid N$. We calculate $\tilde a_c$ for $c=N/(pq)$
        and $\E\setminus\{1\}$:
        \begin{align*}
          \tilde a_c 
          &= a_c \,- \frac{\varphi(N)}{\varphi(N/(N,N/(pq)))}
          \mu(N/(N,N/(pq))) \\
          &= 0 - \frac{\varphi(N)}{\varphi(pq)}\mu(pq) 
          = - \frac{\varphi(N)}{(p-1)(q-1)} <0.
        \end{align*}
    \end{enumerate}
  \end{proof}

  The main part of the proof of Theorem \ref{thm:Roots} is the
  following claim, which we show by induction:
    Let $N\in\N$, $\D(N)$ the set of all divisors of $N$ and $\E\subset
    \D(N)$. If
    \begin{equation}\label{eq:claim}
      a_c= \sum_{d\in\E} \frac{\varphi(N/d)}{\varphi(N/(cd,N))} \mu(N/(cd,N))
      \geq 0
    \end{equation}
    for all $c\in\D(N)$, then $\E$ is a filter in $\D(N)$ as is Definition
    \ref{defi:Filter}.

    \begin{defi}\label{defi:Ep}
      Let $N\in\N$ and $\E\subset\D(N)$. For a prime factor $p\mid N$ we
      define a new set of divisors of $N/p$, namely
      \begin{equation*}
        \E_p=\{d/p\mid d\in\E,~p\mid d\}\subset\D(N/p).
      \end{equation*}
      For $c\in\D(N/p)$ we denote by $a_c^p$ the corresponding sum over $\E_p$,
      i.e.
      \begin{equation*}
        a_c^p =\sum_{(d/p)\in\E_p}
        \frac{\varphi\left(\frac{N/p}{d/p}\right)}
        {\varphi\left(\frac{N/p}{(N/p,cd/p)}\right)}
        \mu\left(\frac{N/p}{(N/p,cd/p)}\right).
      \end{equation*}
      We will in the following only consider $\E_p$ for all $p$, such that
      there exists any $d\in \E$ with $p|d$, so $\E_p$ is not empty.
    \end{defi}

    We use this as induction step $N/p\mapsto N$. We first wish to prove that
    $a_c\geq 0$ for $\E$ implies $a_c^p\geq 0$ for $\E_p$ and all
    $c\in\D(N/p)$. 
    \begin{lem}\label{lem:acp}
      For any $p|N$ with $\E_p$ we get for all $c\in\D(N/p)$:
      \begin{equation*}
        a_c^p = \begin{cases}
                  a_c, & \text{if }pc\mid N/p,\\
          \frac1p((p-1)a_{c}+a_{pc}) & \text{if }pc\nmid N/p.
                \end{cases}
      \end{equation*}
      In particular, $a_c\geq 0$ for all $c\in\D(N)$
      implies $a_c^p\geq 0$ for all $c\in \D(N/p)$.
    \end{lem}

  \begin{proof}
    For $\E$ and $p\in\D(N)$ such that $p$ divides at least one $d\in \E$, the
    set $\E_p$ is non-empty. We calculate the value of $a_c^p$ for all 
    $c\in\D(N/p)$:
    \begin{align*}
      a_c^p &=\sum_{(d/p)\in\E_p}
      \frac{\varphi\left(\frac{N/p}{d/p}\right)}
      {\varphi\left(\frac{N/p}{(N/p,cd/p)}\right)}
      \mu\left(\frac{N/p}{(N/p,cd/p)}\right) \\
      &=\sum_{\substack{d\in\E \\ p\mid d}}
      \frac{\varphi\left(\frac{N}{d}\right)}
      {\varphi\left(\frac{N}{(N,cd)}\right)}
      \mu\left(\frac{N}{(N,cd)}\right) \\
      &= a_c - 
      \underset{=:a_c^{p'}}{\underbrace{
      \sum_{\substack{d\in\E \\ p\nmid d}}
      \frac{\varphi\left(\frac{N}{d}\right)}
      {\varphi\left(\frac{N}{(N,cd)}\right)}
      \mu\left(\frac{N}{(N,cd)}\right).}}
    \end{align*}
    
    For $n\in\N$ and prime number $p$ let $\nu_p(n)\geq 0$ the maximal
    $p$-part of $n$, i.e. $p^{\nu_p(n)}\mid n$ and $p^{\nu_p(n)+1}\nmid n$.
    Let $k=\nu_p(N)$. If $\nu_p(c)\leq k-2$, i.e. in particular it is $k\geq
    2$, then $\nu_p(N/(cd,N)) \geq 2$ for $d\in\E$, $p\nmid d$. Thus, we have
    $\mu(N/(cd,N))=0$ and $a_c^p=a_c$ in the case $pc\mid (N/p)$. For
    $\nu_p(c)=k-1$ we have
    \begin{align*}
      a_c^{p'} 
      &= \sum_{\substack{d\in\E \\ p\nmid d}}
      \frac{\varphi\left(\frac{N}{d}\right)}
      {\varphi\left(\frac{N}{(N,cd)}\right)}
      \mu\left(\frac{N}{(N,cd)}\right) \\
      &= \sum_{\substack{d\in\E \\ p\nmid d}}
      \frac{\varphi(p^k)\varphi\left(\frac{N/p^k}{d}\right)}
      {\varphi(p)\varphi\left(\frac{N/p^k}{(N/p^k,dc/p^{k-1})}\right)}
      \underset{=-1}{\underbrace{\mu(p)}}
      \mu\left(\frac{N/p^k}{(N/p^k,dc/p^{k-1})}\right) \\
      &= -p^{k-1} 
      \underset{=:X(c)}{\underbrace{
          \sum_{\substack{d\in\E \\ p\nmid d}}
      \frac{\varphi\left(\frac{N/p^k}{d}\right)}
      {\varphi\left(\frac{N/p^k}{(N/p^k,dc/p^{k-1})}\right)}
      \mu\left(\frac{N/p^k}{(N/p^k,dc/p^{k-1})}\right).}}
    \end{align*}
    Assume now $\nu_p(c)=\nu_p(N)=k$. We write $N'=N/p^k$ and $c'=c/p^k$, then
    we get
    \begin{align*}
      a_c &= \sum_{d\in\E}
      \frac{\varphi\left(\frac{N'p^k} {d}\right)}
      {\varphi\left(\frac{N'p^k} {(N'p^k,c'p^kd)} \right)}
      \mu\left(\frac{N'p^k} {(N'p^k,c'p^kd)} \right) \\
      &= \sum_{\substack{d\in\E\\p\nmid d}}
      \frac{\varphi(p^k)\varphi\left(\frac{N'} {d}\right)}
      {\varphi\left(\frac{N'} {(N',c'd)} \right)}
      \mu\left(\frac{N'} {(N',c'd)} \right) 
      + \underset{=a_{(c'p^{k-1})}^p}{\underbrace{
          \sum_{\substack{d\in\E\\p\mid d}}
      \frac{\varphi\left(\frac{N'} {d}\right)}
      {\varphi\left(\frac{N'} {(N',c'd)} \right)}
      \mu\left(\frac{N'} {(N',c'd)} \right)}} \\
      &= (p-1)p^{k-1} X(c) + a_{(c'p^{k-1})}^p.
    \end{align*}
    We combine the two expressions for $a_c$, $c=c'p^k$.
    Then, $c/p=c'p^{k-1}\mid N/p$ and $X(c'p^{k-1})=X(c'p^{k})=X(c)$. Since $x
    a_{(c'p^{k-1})} + y a_{(c'p^k)}\geq 0$ for $x,y\geq 0$ we get from
    \begin{align*}
      x a_{(c'p^{k-1})} + y a_{(c'p^k)} 
      &= x\left( a_{(c'p^{k-1})}^p - p^{k-1} X(c)\right)
      + y\left( (p-1)p^{k-1} X(c) + a_{(c'p^{k-1})}^p\right) \\
      &= p a_{(c'p^{k-1})}^p, \tag{$x=p-1,~ y=1$}
    \end{align*}
    and this proves the lemma.
  \end{proof}

  We now conclude by induction that $\E$ is a filter if all possible reductions
  $\E_p$, $p\mid N$, are filters. Under this assumption, it follows that no
  $\E_p$ is empty: Let $p\mid N$ such that $p\mid d$ for some $d\in \E$, then
  $\E_p$ is not empty, hence equals $(e_p)$ for some $e_p$. Since $\E_p$ is a
  filter, we have $N/p\in\E_p$, and hence $N=p\cdot N/p\in\E$. As induction
  step, we use the following lemma.

  \begin{lem}\label{lem:induction}
    Let $\E\subset\D(N)$ and for all $p\mid N$ the set $\E_p$, defined as in
    Definition \ref{defi:Ep}, a filter, namely $(e_p)=(e_p)_{N/p}$ for some
    $e_p\mid N/p$.  Then either there exist a prime $p\mid N$ with $p\mid e_q$
    for all $p\neq q$ or it is $e_p=1$ for all $p$.\\ In the first case we
    have $\E=(pe_p)$ or $\E=(pe_p)\cup \{1\}$. In the second case we have
    $\E=(1)=\D(N)$ or $\E=(1)\setminus \{1\}$.
  \end{lem}

  \begin{proof}
    Assume, there exist $q'$ with $e_{q'}\neq 1$ and $p\mid e_{q'}$. Then
    $p\mid e_q$ for all $q\neq p$. We prove this by contradiction, then if
    $p\nmid e_q$ for some $q$, then $q'e_q\in\E_q$ (since $\E_q$ is a filter)
    and hence $qq'e_q\in \E$ (by definition of $\E_q$). Then $qe_q\in\E_{q'}$
    and $e_{q'}\mid qe_q$, hence a contradiction to $p\nmid e_q$. This proves
    the first part of the lemma.

    We now prove the consequences in the two cases. Firstly, we assume it
    exist $p$ with $p\mid e_q$ for all $q\neq p$. Let $x\in(pe_p)$, then
    $e_p\mid x/p$, and since $\E_p$ is a filter, $x/p\in\E_p$. Thus we have
    $x=p\cdot x/p\in\E$, i.e. $(pe_p)\subset\E$. Let now $x\in \E$. If $p\mid
    x$ we have $x/p\in\E_p$, hence $e_p\mid x/p$ and therefore $pe_p\mid x$. If
    $p\nmid x$, then $x=1$ or it exists $q$ with $q\mid x$. In this case is
    $x/q\in\E_q$ and $e_q\mid x/q$, which is a contradiction to $p\nmid x$.
    This proves $(pe_p)=\E\setminus\{1\}$ (which may be equal to $\E$).\\
    In the case $e_p=1$ for all $p$, we have $(p)\subset\E$ for all $p$. Since
    it is $\bigcup_{p\in\D(N)}(p)=\D(N)\setminus\{1\}$, this proves the
    assertion.
  \end{proof}

  We can now conclude the proof of the claim of \eqref{eq:claim}. Let $\E$
  such that 
    \begin{equation*}
      a_c= \sum_{d\in\E} \frac{\varphi(N/d)}{\varphi(N/(cd,N))} \mu(N/(cd,N))
      \geq 0
    \end{equation*}
  for all $c\in\D(N)$, then $a_c^p\geq 0$ for all $\E_p$ and $c\in\D(N/p)$ by
  Lemma \ref{lem:acp}. By induction, all $\E_p$ are filters, namely $(e_p)$
  for some $e_p\mid N/p$. Then, by Lemma \ref{lem:induction}, we have that
  $\E$ is a filter $(e)$ for some $e\mid N$ or $(e)\cup \{1\}$ for some $e\neq
  1$ or $(1)\setminus\{1\}$. By Lemma \ref{lem:modification}, the last cases
  are only possible for $N=p^n$ and $e=p$. In this cases $\E$ is a filter as
  well.  This proves the claim \eqref{eq:claim} and hence concludes the proof
  of Theorem \ref{thm:Roots}.

\section{Fourier pairs of $\{0,1\}$-matrices}\label{sec:Fourier}
  In the following, we consider idempotents $\eps/N$ of the group algebra
  $\C[\Z_N\times\Z_N]$, i.e. 
  \begin{equation}\label{eq:eps-idem}
    \eps/N \cdot \eps/N =\eps/N.
  \end{equation}
  If we introduce a basis $\{g^i\otimes g^j \mid 0\leq i,j<N\}$, with
  $\car{g}=\Z_N$, and write $\eps=\sum_{i,j} \eps_{ij} g^i\otimes g^j$, 
  equation \eqref{eq:eps-idem} translates to
  \begin{equation*}
    \frac 1{N^2} \sum_{i',i''}\sum_{j',j''} \eps_{i'j'} \eps_{i''j''}
    (g^{i'}\otimes g^{j'})(g^{i''}\otimes g^{j''})
    = \frac 1N \sum_{i,j} \eps_{ij} g^i\otimes g^j,
  \end{equation*}
  and by comparing coefficients, we get
  \begin{equation}\label{eq:eps-sum}
    \frac 1{N^2} \sum_{i'+i''=i}\sum_{j'+j''=j} \eps_{i'j'} \eps_{i''j''}
    = \frac 1N \eps_{ij}.
  \end{equation}
  Let $\xi=\xi_N$ be a primitive $N$-th root of unity and $\{e_k=1/N
  \sum_{r=0}^N \xi^{kr}g^r \mid 0\leq k< N\}$ be the set of primitive
  idempotents of the group algebra $\C[\Z_N]$. Then $\{e_k\otimes e_l\}$ is the
  set of primitive idempotents of $\C[\Z_N\times\Z_N]$ and we can express
  $\eps/N$ as sum of these primitive idempotents:
  $\eps/N=\sum_{k,l}\bar\eps_{kl} e_k\otimes e_l$ with
  $\bar\eps_{kl}\in\{0,1\}$ for all $0\leq k,l<N$. This leads to 
  \begin{equation}\label{eq:Fourier}
    \eps_{ij} = \frac 1{N} \sum_{k,l}\bar\eps_{kl} \xi^{ik+jl},
  \end{equation}
  which means, that the matrix $\eps=(\eps_{ij})_{ij}$ is the discrete
  Fourier transformation of the $\{0,1\}$-matrix $\bar\eps$. This
  considerations lead to the following problem.
  \begin{prob}
    We wish to determine all idempotents $\eps/N$ of $\C[\Z_N\times \Z_N]$,
    such that $\eps=(\eps_{ij})_{ij}$ is $\{0,1\}$-matrix, or equivalent, all
    Fourier pairs of $\{0,1\}$-matrices $\eps$ and $\bar\eps$.
  \end{prob}

  \begin{expl}
    We consider some examples of Fourier transformed matrices, where $\eps$ is
    not necessarily a $\{0,1\}$-matrix.
    \begin{enumerate}[(i)]
      \item Let $\bar\eps$ the matrix with $\bar\eps_{00}=1$ and
        $\bar\eps_{kl}=0$ otherwise. Then $\eps_{ij}=1/N$ for all $j,j$.
      \item Let $\bar\eps$ the matrix with $\bar\eps_{0l}=1$ for all $l$ and
        $\bar\eps_{kl}=0$ otherwise. Then $\eps_{i0}=1$ for all $i$ and
        $\eps_{ij}=0$ otherwise.
      \item Let $\bar\eps$ the matrix with $\bar\eps_{kk}=1$ for all $k$ and
        $\bar\eps_{kl}=0$ otherwise. Then $\eps_{ij}=1$ for all $i,j$ with
        $i+j\equiv 0\mod N$ and $\eps_{ij}=0$ otherwise. We give the matrices
        explicitly for $N=4$:
        \begin{equation*}
          \bar\eps=
          \begin{pmatrix} 
            1&0&0&0 \\
            0&1&0&0 \\
            0&0&1&0 \\
            0&0&0&1
          \end{pmatrix}
          \quad \longrightarrow \quad
          \eps=
          \begin{pmatrix} 
            1&0&0&0 \\
            0&0&0&1 \\
            0&0&1&0 \\
            0&1&0&0
          \end{pmatrix}.
        \end{equation*}
    \end{enumerate}
  \end{expl}

  The following theorem completely solves this problem, relying heavily on
  our main Theorem \ref{thm:Roots}:
  \begin{thm}\label{thm:eps-Solutions}
    All idempotents $\eps/N$ of the group algebra $\C[\Z_N\times\Z_N]$ with
    $\eps_{ij} \in\{0,1\}$, or equivalently all discrete Fourier pairs
    $\eps,\bar\eps$ of $\{0,1\}$-matrices are either
    \begin{equation}\label{eq:eps}
      \eps_{ij} = \delta_{(\frac Nd\mid i)} \delta_{(d\mid j-t\frac i{N/d})},
    \end{equation}
    for a unique $d\mid N$ and $0\leq t\leq d-1$ or
    they are trivial $\eps=\bar\eps=0$. 
  \end{thm}

  Before we proceed to the proof of the theorem, we give another Example.
  \begin{expl}
    Let $N=12$, $d=3$ and $t=2$, then $\eps$ as in \eqref{eq:eps} is given by
    \begin{equation*}
      \bordermatrix{
        {}& 0&  & t&  & d&  &  &  &  &  &  &   \cr
      ~~0 & 1& 0& 0& 0& 1& 0& 0& 0& 1& 0& 0& 0 \cr
          & 0& .& .& .& .& .& .& .& .& .& .& . \cr
          & 0& .& .& .& .& .& .& .& .& .& .& . \cr
      N/d & 0& 0& 1& 0& 0& 0& 1& 0& 0& 0& 1& 0 \cr
          & 0& .& .& .& .& .& .& .& .& .& .& . \cr
          & 0& .& .& .& .& .& .& .& .& .& .& . \cr
~~\bar d  & 1& 0& 0& 0& 1& 0& 0& 0& 1& 0& 0& 0 \cr
          & 0& .& .& .& .& .& .& .& .& .& .& . \cr
          & 0& .& .& .& .& .& .& .& .& .& .& . \cr
          & 0& 0& 1& 0& 0& 0& 1& 0& 0& 0& 1& 0 \cr
          & 0& .& .& .& .& .& .& .& .& .& .& . \cr
          & 0& .& .& .& .& .& .& .& .& .& .& . \cr
      }.
    \end{equation*}
  \end{expl}

  \begin{proof}
    By the Fourier transformation \eqref{eq:Fourier}, we have $\eps_{00} =
    \frac 1N\sum_{k,l}\bar\eps$.  Since $\eps$ and $\bar\eps$ are
    $\{0,1\}$-matrices, this gives either $\bar\eps_{kl}=0$ for all $k,l$ or
    $\sum_{k,l}\bar\eps_{kl}=N$. In the first case we have $\eps=0$ as well.
    Applying the same argument to the dual Fourier transformation, 
    \begin{equation}\label{eq:dualFourier}
      \bar\eps_{kl} = 1/N\sum_{ij} \eps_{ij} \xi^{-(ik+jl)},
    \end{equation}
    we get $\sum_{i,j}\eps_{ij}=N$ in the second case.

    Assume in the following $\eps\neq 0$. We now calculate the row-, resp.
    column-sums of the matrices $\eps$ and $\bar\eps$. Let $a_k$ the $k$-th
    row sum of $\eps$ and $a'_l$ the $l$-th column-sum for $\eps$, and $\bar
    a_k$, $\bar a'_l$ the according sums of $\bar\eps$. Then
    \begin{equation}
       a_k = \sum_j \eps_{kj} =\frac 1N \sum_{l,i,j} \bar\eps_{ij} \xi^{ik+jl}
       =\frac 1N \sum_i \xi^{ik} \sum_j \bar\eps_{ij}
       \underset{=N\delta_{(j=0)}}{\underbrace{\sum_l \xi^{jl}}}
       = \sum_i \xi^{ik}\bar\eps_{i0}
    \end{equation}
    Since the row sum $a_k$ is a non-negative integer, we get by Theorem
    \ref{thm:Roots}, that $\{i\mid \bar\eps_{i0}=1\}$ is a subgroup of $\Z_N$,
    thus it exist $d\mid N$ with $\bar\eps_{i0} = \delta_{(d\mid i)}$.
    Analogously, we get by calculating the column sum $a'_l$ that $d'\mid N$
    exist, such that $\bar\eps_{0j} = \delta_{(d'\mid j)}$. As a consequence,
    we get $a_k=N/d$ for $k$ being an $N/d$-multiple and $a'_l=N/d'$ for $l$
    being an $N/d'$-multiple, and the other row, resp. column sums being $0$.
    We now calculate the row and column sums of $\bar\eps$ using the dual
    transformation \eqref{eq:dualFourier}. 
     This gives, again by application of Theorem
    \ref{thm:Roots}, that there exist $\bar d,\bar d'\mid N$ such that
    $\eps_{i0}=\delta_{(\bar d\mid i)}$ and $\eps_{0j}=\delta_{(\bar d'\mid
    j)}$, and $\bar a_k=N/\bar d$ for $k$ being $N/\bar d$-multiple and $\bar
    a'_l=N/\bar d'$ for $l$ being $N/\bar d'$-multiple. Since
    $N/d=a_0=\sum_{j}\eps_{0j}=N/\bar d'$, we get $d=\bar d'$. Analogously, we
    get $\bar d=d'$. Since only the $N/d$-th rows have entries $1$ and
    $\eps_{i0}=\delta_{(\bar d\mid i)}$, we get $N/d\mid \bar d=d'$, hence
    $N\mid dd'$. 

    The case $N=dd'$, i.e. $d'=N/d$ corresponds to the
    non-shifted solution of \eqref{eq:eps-sum}, since in this case there
    maximal $N/d\cdot N/d'=N$ entries $1$. 
    
    We consider now a solution with $N<dd'$ and show, that a suitably shifted
    version of this is also a solution of \eqref{eq:eps-sum} with smaller
    $dd'$. The claim then follows by induction over $dd'$.  For a solution
    $\eps$ of \eqref{eq:eps-sum} the shifted
    matrix, defined by
    \begin{equation*}
      \eps_{ij}^{[t]} \df \begin{cases}
        \eps_{i,j-t\frac i{N/d}}, & N/d\mid i,\\
        \eps_{i,j}=0, & \text{otherwise}.
      \end{cases}
    \end{equation*}
    for some $0\leq t\leq d-1$, is also a solution of \eqref{eq:eps-sum}. This
    follows easily by inserting $\eps_{ij}^{[t]}$ in \eqref{eq:eps-sum}, since
    the shift gives only a new ordering of the summands. Let now $\eps$ be a
    solution with $dd'>N$. We now want to shift this $\eps$ in a way, that no
    $1$ entries in the first column are moved, i.e. the $d'$-th rows are 
    shifted by multiple of $d$, and some of the other rows are shifted, such
    that $\eps^{[t]}$ has at least one $1$-entry more in the first column,
    than $\eps$:

    Consider the $N/d$-row. By hypothesis, $N/d<d'$, hence
    $\eps_{N/d,0}=0$, but $\eps_{N/d,t}=1$ for some $t$, since the row sum
    $a_{N/d}=N/d>0$. Since the column sum $a'_t\neq 0$, we have $N/d'\mid t$,
    hence $d\mid t\cdot dd'/N$.
    Thus the shifted solution $\eps_{ij}^{[-t]}$ has in the $0$-column
    still $\eps_{i0}=1$ for $d'\mid i$ and it has now additionally
    $\eps_{N/d,0}=1$. The expression $dd'$ for $\eps^{[t]}$ has to be strictly
    smaller than $dd'$ for $\eps$, this reduces the claim by induction to the
    unshifted case $dd'=N$, which has been solved above.
  \end{proof}
 
\section {A system of equations and $R$-matrices of quantum groups}
  The following system of equations for an abelian group $G$ arises as a
  necessary condition on the the element $R_0\in\C[\Lambda\times \Lambda]$ in
  Lusztig's ansatz for $R$-matrices for a quantum group $U_q(\g)$ with
  coradical $\C[\Lambda]$. In this application, the abelian group $G$ will be
  the fundamental group of $\g$, and hence cyclic except for $\g=D_{2n}$. We
  will not discuss this matter further, but refer the reader to our respective
  paper \cite{LN14}. Note Remark \ref{rem:DL}.

  \begin{defi}\label{GroupEquations}
    For an abelian group $G$ we define a set of $2|G|^2+2$ quadratic equations
    in $|G|^2$ formal complex variables $g(x,y)$ indexed by $x,y\in
    G$:
    \begin{align}
      g(x,y) &= \sum_{y_1+y_2=y} g(x,y_1)g(x,y_2),\label{grpeq01}\\
      g(x,y) &= \sum_{x_1+x_2=x} g(x_1,y)g(x_2,y),\label{grpeq02}\\
      1 &= \sum_{y\in G} g(0,y),\label{grpeq03}\\
      1 &= \sum_{x\in G} g(x,0).\label{grpeq04}
    \end{align}
  \end{defi}
  As a side remark, note that these equations are a subset of the equations for
  a Hopf pairing $g:\C^G\otimes \C^G\to \C$, but it allows for significantly
  more solutions containing $0$'s, as the next theorem shows. The result of this
  article is in some sense, that $g$ is still a pairing on a pair of subgroups.
  \begin{thm}\label{solutionsgrpeq}
    Let $G$ be an abelian group and $H_1,H_2$ subgroups with equal cardinality 
    $|H_1|=|H_2|=d$ (not necessarily isomorphic!).
    Let $\omega\colon H_1\times H_2\to\C^{\times}$ be a pairing of groups.
    Here, the group $G$ is written additively and $\C^{\times}$
    multiplicatively, thus we have $\omega(x,y)^d = 1$ for all $x\in H_1,
    y\in H_2$. Then the assignment
    \begin{equation}\label{g-solution}
      g\colon G\times G\to \C, ~ (x,y) \mapsto \frac 1d 
      \, \omega(x,y)\delta_{(x\in H_1)}\delta_{(y\in H_2)}
    \end{equation}
    is a solution of the equations \eqref{grpeq01}-\eqref{grpeq04} for $G$.
  \end{thm}

  \begin{proof}
    The claim follows by straightforward calculations:
    \begin{align*}
      \sum_{y_1+y_2=y}g(x,y_1)g(x,y_2)
      &=\left(\frac 1d \right)^2\sum_{y_1+y_2=y} 
        \omega(x,y_1)\omega(x,y_2) 
        \delta_{(x\in H_1)}\delta_{(y_1\in H_2)}\delta_{(y_2\in H_2)} \\
      &=\left(\frac 1d\right)^2 \sum_{y_1+y_2=y}
        \omega(x,y_1+y_2)
        \delta_{(x\in H_1)}\delta_{(y_1\in H_2)}\delta_{(y_2\in H_2)} \\
      &=\left(\frac 1d\right)^2 |H_2|\, \omega(x,y) \delta_{(x\in
      H_1)}\delta_{(y\in H_2)}  
      = g(x,y).\\
      \sum_{y\in G} g(0,y) 
      & = \frac 1d\sum_{y\in G} \omega(0,y)\delta_{(y\in H_2)}
      = \frac 1d \sum_{y\in H_2}1
      = 1.
    \end{align*}
  \end{proof}

  \begin{que}
    Are these all solutions of the equations
    \eqref{grpeq01}-\eqref{grpeq04}?
  \end{que}

  As an application of the theorems proved in this paper we will below
  positively  answer this question for a {\bf cyclic group} $G$. We would
  actually hope to  completely resolve the question with the combinatorial
  results of this  article.

  \begin{rem}\label{rem:DL}
    For the application in quantum groups,
    the only non-cyclic case of interest is $\Z_2\times \Z_2$ (the fundamental
    group of the Lie algebra $\g=D_{2n}$), which can be checked explicitly to
    hold as well. Most other Lie algebras have $G=\Z_1,\Z_2,\Z_3,\Z_4$ \\

    It is quite remarkable that the only highly nontrivial case solved with 
    this articles result is hence the Lie algebra $A_n$ with $G=\Z_{n+1}$,
    which depends highly on the prime divisors of $n+1$. This is due to the
    {\bf unusually large center} $\Z_{n+1}$ of the algebraic group $SL_{n+1}$,
    which makes it notoriously hard to deal with (e.g. in Deligne-Lusztig
    theory). We hope that the technical tools developed in this article might be
    useful in addressing such issues.
  \end{rem}

  \begin{expl}\label{cyclicSolutions}
    Let $G=\Z_N$ and consider for any divisor $d|N$ the unique
    subgroup $H=\frac Nd\Z_N\cong \Z_d$ of $G$ of order $d$. By Theorem
    \ref{solutionsgrpeq} we have for any pairing $\omega\colon H\times
    H\to \C^{\times}$ the function $g$ as in \eqref{g-solution}
    as a solution of the equations \eqref{grpeq01}-\eqref{grpeq04}.\\
    We give the solution explicitly: For $H=\car{h}$, $h\in \frac nd\Z_n$, we
    define a pairing $\omega\colon H\times H\to \C^{\times}$ by
    $\omega(h,h)=\xi$ with $\xi$ a $d$-th root of unity, not necessarily
    primitive.\\
    Thus the general solution ansatz in Lemma \ref{solutionsgrpeq}
    translates for cyclic groups $G$ to
    \begin{equation}
      g\colon G\times G,~ (x,y)\mapsto \frac 1d \, \xi^{\frac{xy}{(N/d)^2}}
      \delta_{(\frac Nd\mid x)}\delta_{(\frac Nd\mid y)}.
    \end{equation}
  \end{expl}

  \begin{thm}\label{allsolutionsgrpeq}
    For $G=\Z_N$ the solutions given in Lemma \ref{solutionsgrpeq} (and worked
    out in this case in example \ref{cyclicSolutions}), are
    in fact all solutions to the system of equations
    \eqref{grpeq01}-\eqref{grpeq04}.
  \end{thm}

  \begin{proof}
    \begin{enumerate}[(a)]
      \item The proof is an application of Theorem \ref{thm:eps-Solutions}, which
      follows from the main Theorem \ref{thm:Roots}. Let
      $g:\Z_N\times \Z_N\to\C$ be
        a solution of the equations \eqref{grpeq01}-\eqref{grpeq04}. We write
        shortly $g_{ij}$ for $g(i,j)$, $0\leq i,j\leq N-1$.
        Let $\Z_N=\car{g}$, then 
        \begin{equation*}
          \sum_{j',j''} (g_{ij'}x^{j'}) (g_{ij''}x^{j''}) = \sum_j g_{ij} x^i
        \end{equation*}
        for all $i$ by \eqref{grpeq01}, hence $\sum_j g_{ij} x^i$ is an
        idempotent in $\C[\Z_N]$. Let $\xi=\xi_N$ be a primitive $N$-th root
        of unity, then primitive idempotents of $\C[\Z_N]$ are all of the form
        $e_k=\frac 1N\sum_{r=0}^{N-1} \xi^{kr}x^r$. Thus, we have $\sum_j
        g_{ij} x^i =\sum_k \eps_{ik}e_k$ for $\eps_{ik}\in\{0,1\}$ for all
        $i$, and therefore
        \begin{equation*}\label{eq:gij}
          g_{ij} = \frac 1N\sum_{k=1}^{N-1} \eps_{ik} \xi^{jk}
        \end{equation*}
        for $\{0,1\}$-matrix $\eps=(\eps_{ik})$. By inserting this in
        \eqref{grpeq02}, $\sum_{i'+i''=i}g_{i'j}g_{i''j} = g_{ij}$, we get
        \begin{equation*}
          \frac1{N^2} \sum_{i'+i''=i}\sum_{k',k''} \eps_{i'k'} \eps_{i''k''}
          \xi^{(k'+k'')j} 
          = \frac 1N \sum_{k} \eps_{ik} \xi^{kj}.
        \end{equation*}
        By comparing the coefficients on both sides we get
        \begin{equation}
          \frac 1{N^2} \sum_{i'+i''=i}\sum_{k'+k''=k} \eps_{i'k'}
          \eps_{i''k''} = \frac 1N \eps_{ik},
        \end{equation}
        which is equation \eqref{eq:eps-sum}. Thus, $\eps/N$ is an idempotent
        in $\C[\Z_N\times\Z_N]$ and $\eps$ is $\{0,1\}$-matrix and we can
      apply Theorem \ref{thm:eps-Solutions}. We have 
      \begin{equation*}
        \eps_{ij} = 
        \begin{cases}
          \delta_{(\frac Nd \mid i)} \delta_{(d\mid j-t\frac
          i{N/d})}, & \text{if } N/d\mid i, \\
          0, &\text{otherwise,}
        \end{cases}
      \end{equation*}
      for some $d\mid N$ and $0\leq t\leq d-1$. We insert in
      \eqref{eq:gij}:
      \begin{align*}
        g_{ij} 
        &= \frac 1N \sum_{k=0}^{N-1} \delta_{(\frac Nd \mid i)} 
           \delta_{(d\mid k-t\frac i{N/d})} \xi^{jk} \\
        &= \frac 1N \delta_{(\frac Nd \mid i)} \sum_{k'=0}^{N/d-1}
           \xi^{j(t\frac i{N/d} + dk')}
           \tag{$k=t\frac i{N/d} + dk',~k'=0,\ldots,d-1$} \\
        &= \frac 1d \left(\xi^{N/d}\right)^{t\frac i{N/d} \frac j{N/d}}
           \delta_{(\frac Nd \mid i)} 
           \underset{=\delta_{\frac Nd\mid j}}
           {\underbrace{\frac{1}{N/d} \sum_{k'=0}^{N/d-1} (\xi^d)^{jk'}}}
      \end{align*}
      Thus, $g$ is the solution given already in Example
      \ref{cyclicSolutions}, which was the explicitly worked out case of
      Lemma \ref{solutionsgrpeq} for $G$ cyclic.
    \end{enumerate}
  \end{proof}



\begin{thebibliography}{12345}
  \bibitem[Bou66]{Bou66} N. Bourbaki, \emph{General topology, part 1},
    Hermann, Paris and Addison-Wesley (1966).
  \bibitem[LN14]{LN14} S. Lentner, D. Nett: \emph{New $R$-matrices for small
    quantum groups}, Preprint,
    \href{http://arxiv.org/abs/1409.5824}{\texttt{arXiv:1409.5824}} (2014).
  \bibitem[Lus93]{Lus93} G. Lusztig, \emph{Introduction to quantum
    groups}, Birkh\"auser (1993).
\end{thebibliography}
\end{document}